\documentclass[a4paper,12pt]{article}
\usepackage{amssymb}
\usepackage{graphicx}
\usepackage{amsmath,amsthm,amsfonts,amssymb,graphicx}

\newtheorem{theorem}{Theorem}[section]
\newtheorem{lemma}[theorem]{Lemma}

\newtheorem{prop}[theorem]{Proposition}
\newtheorem{conjecture}{Conjecture}
\newtheorem{problem}[conjecture]{Problem}
\newcommand{\id}{e}

\newcommand{\R}{{\mathbb R}}
\newcommand{\rodl}{R\"odl}
\newcommand{\kriz}{K\v{r}\'i\v{z}}
\newcommand{\erdos}{Erd\H os}

\newcommand{\mono}{monochromatic}
\newcommand{\gv}{\vec{g}}
\newcommand{\sy}{\times}
\newcommand{\gih}{\gv\sy_I h}
\newcommand{\Wlog}{wlog}
\newcommand{\piv}{\vec{\pi}}
\newcommand{\pis}{\piv\sy_I\sigma}
\newcommand{\N}{{\mathbb N}}
\newcommand{\p}{\mathcal{P}}
\newcommand{\Q}{\mathcal{Q}}
\newcommand{\ap}{arithmetic progression}
\newcommand{\fp}{subtransitive}

\title{Transitive Sets in
Euclidean Ramsey Theory}

\author{Imre Leader\footnote{Department
of Pure Mathematics and Mathematical Statistics,
Centre for Mathematical Sciences,
Wilberforce Road,
Cambridge CB3 0WB,
England.} \footnote{{\tt I.Leader@dpmms.cam.ac.uk}}\and 
Paul A.~Russell\footnotemark[1] \footnote{{\tt P.A.Russell@dpmms.cam.ac.uk}}
\and Mark Walters \footnote{School of Mathematical Sciences, Queen Mary,
University of London, London E1 4NS, England} 
\footnote{\tt m.walters@qmul.ac.uk}}

\begin{document}
\maketitle

\begin{abstract}
A finite set $X$ in some Euclidean space $\R^n$ is called {\it Ramsey}
if for any $k$ there is a $d$ such that whenever $\R^d$ is $k$-coloured
it contains a \mono\ set congruent to $X$.  This notion was introduced
by \erdos, Graham, Montgomery, Rothschild, Spencer and Straus,
who asked if a set is Ramsey if and only if it is
{\it spherical,} meaning that it lies on the surface of a sphere.
This question (made into a conjecture by Graham) has dominated subsequent 
work in Euclidean Ramsey theory.

In this paper we introduce a new conjecture regarding which sets are Ramsey;
this is the first ever `rival' conjecture to the conjecture above.
Calling a finite set {\it transitive} if its symmetry group acts
transitively---in other words, if all points of the set look the same---our
conjecture is that the Ramsey sets are precisely the transitive
sets, together with their subsets.
One appealing feature of this conjecture is that it reduces
(in one direction) to a purely combinatorial statement.  We give this statement
as well as several other related conjectures.  We also prove the first
non-trivial cases of the statement.

Curiously, it is far from obvious that our new conjecture is genuinely 
different from the old.  We show that they are indeed different by proving
that not every spherical set embeds in a transitive set.  This result
may be of independent interest.
\end{abstract}

\begin{section}{Introduction}\label{intro}
Euclidean Ramsey theory originates in the sequence of papers \cite{egmrss},
\cite{egmrss2} and \cite{egmrss3} by \erdos, Graham, Montgomery,
Rothschild, Spencer and Straus.
A finite set $X$ in some Euclidean space $\R^n$ 
is said to be {\it Ramsey} if for every 
positive integer $k$ there exists a positive integer $d$ such that
whenever $\R^d$ is $k$-coloured it must contain a \mono\ subset congruent to
$X$.  For example, it is easy to see that (the set of vertices of) the
$r$-dimensional
regular simplex is Ramsey:  $\R^{kr}$ contains a collection of $kr+1$
points with each pair at distance $1$, and whenever $\R^{kr}$ is $k$-coloured
some $r+1$ of these points must be the same colour.  On the other hand,
the subset $\{0,1,2\}$ of $\R$ is not Ramsey.  
To see this, observe that a copy of $\{0,1,2\}$ is simply a collection
of three collinear points $x$, $y$, $z\in\R^n$ with $\|z-x\|=2$ and
$y={1\over2}(x+z)$.  It follows from the parallelogram law that
$\|y\|^2={1\over2}(\|x\|^2+\|z\|^2)-1$.  It is now easy to write down
a $4$-colouring of $\R^n$ with no \mono\ copy of $\{0,1,2\}$ by colouring
each point $u$ according to its distance $\|u\|$ from the origin---for
example by $c(u)=\lfloor\|u\|^2\rfloor$ (mod $4$).  

So which finite sets are Ramsey?  This question was first considered by
\erdos, Graham, Montgomery, Rothschild, Spencer and Straus in \cite{egmrss}.
They showed that any Ramsey set must be {\it spherical;} that is, it must be
contained in the surface of a sphere.  Their
proof can be viewed as a (much more difficult) extension of the proof
above that $\{0,1,2\}$ is not Ramsey.
So the key question remaining was:  is {\it every} spherical set Ramsey?
It has been widely believed for some time 
that this is in fact the case---indeed,
Graham conjectured this in \cite{conj} and offered \$1000 for its proof.

Various cases of this conjecture have been proved.  In the original paper
\cite{egmrss} it is shown, by means of a product argument, that if $X$ and
$Y$ are Ramsey then so is $X\times Y$.  (Here if $X\subset\R^n$ and
$Y\subset\R^m$ then we regard $X\times Y$ as a subset of $\R^{n+m}$.)
In particular, it follows that
any {\it brick,} meaning the set of vertices of a cuboid in $n$ dimensions, 
is Ramsey.
Further progress has been slow, with each new step requiring significant
new ideas.
Frankl and \rodl\ \cite{triangles} showed that every triangle is Ramsey,
and later \cite{simplices} that
every non-degenerate simplex is Ramsey.
This left the next interesting case as the regular pentagon, which was
shown to be Ramsey by \kriz\ \cite{kriz}. 
\kriz\ actually showed that
any finite set
$X$ which is acted on transitively by a soluble group $G$ of isometries is
Ramsey and, slightly more generally, that $G$ need not itself be soluble
so long as it has a soluble subgroup $H$ whose action on $X$ has at most two
orbits.  
In particular, this implies that all regular polygons are Ramsey
and that the Platonic solids in $3$ dimensions are Ramsey.  
In addition, this result was used by Cantwell~\cite{cantwell} to show
that the $120$-cell, the largest regular polytope in $4$ dimensions,
is Ramsey.
However, the conjecture itself is still wide open.  Indeed,
it is not even known whether or not every cyclic quadrilateral is Ramsey.

Our starting point in this paper is a feature that we have observed to be
common to all known proofs that particular sets are Ramsey.  In each
case, the proof that the set is Ramsey proceeds by first embedding
it in a (finite) {\it transitive} set---a set whose symmetry group acts
transitively---and then making some clever combinatorial argument to show
that this transitive set has the Ramsey property required.  

As an example, let us digress for a moment to see why every triangle embeds
into a transitive set.
Of course, a right-angled triangle is a subset of a
rectangle, and more
generally any acute-angled triangle is a subset of a
cuboid in 3 dimensions.
For a general triangle $ABC$, consider a variable point
$D$ on the perpendicular
dropped from $C$ to $AB$. Choose $D$ such that the angle
$AOB$, where $O$ is the
circumcentre of triangle $ABD$, is a rational multiple
of $\pi$. It follows that 
$A$ and $B$ lie on some regular polygon $\Pi$ (with centre
$O$). Viewing $\Pi$ as
living in the $xy$-plane, we now form a new copy $\Sigma$ of $\Pi$
by translating $\Pi$ 
in the $z$-direction and rotating it about its centre:
the 
resulting `twisted prism' $\Pi \cup \Sigma$ is transitive and
will, if the translation
and rotation are chosen correctly, contain an
isometric copy of $ABC$.

We mention that the transitive sets in which Frankl and \rodl\ \cite{triangles}
embed their triangles are in fact very different.
Indeed, the actual 
machinery for embedding a set into a transitive set and proving this
transitive set Ramsey can differ greatly from paper to paper, 
and the transitive
set can have a {\it much} higher dimension than the original set, but we have
noticed that the transitivity is always present.  Based on this, and some other
facts (discussed below), 
we are led to the following conjecture which asserts that transitivity is
the key property.

\begin{conjecture}\label{main}
A finite set $X\subset\R^n$ is Ramsey if and only if it is (congruent
to) a subset of a finite transitive set.
\end{conjecture}

In general, when we say that a set $X$ is transitive we shall assume
implicitly that $X$ is finite.
For brevity, we shall say that a finite set in $\R^d$ is {\it subtransitive}
if it is congruent to a subset of a 
transitive set in some $\R^n$.  We stress
that this transitive set may have higher dimension than the
original subtransitive set.

We believe that this is a very natural conjecture for various reasons.
To begin with, it turns out that there are several clean statements any of
which would imply that all transitive sets are Ramsey.  Moreover, some
of these statements are purely combinatorial.  In the other direction, the
transitivity of every Ramsey set would give a clear conceptual reason as to why
every Ramsey set must be spherical.  Indeed, it is easy to see that the points
of any transitive set all lie on the surface of the unique smallest closed ball
containing it.

We remark that it is not clear {\it a priori} that Conjecture \ref{main}
is genuinely different from the old conjecture:  
could it be that a finite set $X$ is spherical
if and only if it is \fp?  As we remark above,
every \fp\ set is spherical, but do
there exist spherical sets that are not \fp?  
We show that such sets do indeed exist.
This result may be of independent interest.
One might hope that the proof of this result would give some insight into
showing the existence of a spherical set that is not Ramsey.  However,
we do not see a way to make this work.

The plan of the paper is as follows.
In \S\ref{if}, we consider the `if' direction of our conjecture:
`every subtransitive set is Ramsey', or, equivalently,
`every transitive set is Ramsey'.
Our first step is to remove the geometry, so to say:  we show that this 
direction of
Conjecture \ref{main} would follow from a Hales-Jewett-type statement
for groups (Conjecture \ref{groups}).  Our next step is to remove the group
theory by showing that Conjecture
\ref{groups} can be reformulated as another equivalent Hales-Jewett-type
statement which is purely combinatorial (Conjecture \ref{patterns}).
This shows that our conjecture is, in a certain sense, natural---if Conjecture
\ref{patterns} is true then it will provide a genuine combinatorial reason
why every transitive set must be Ramsey.
In \S\ref{cases} we
prove some of the first non-trivial cases of Conjecture \ref{patterns}.
This is already enough to yield some new examples of Ramsey sets.
In \S\ref{gons}, we show the existence of spherical sets that are not
subtransitive.  More precisely, we show that, for any 
$k\geqslant 16$, almost every
cyclic $k$-gon is not subtransitive.  We remark that our proof is 
non-constructive:  we have no explicit example of such a $k$-gon.
Finally, in \S\ref{conclude}, we briefly discuss the `only if' direction of
Conjecture \ref{main} and give some further problems.

For a general overview of Ramsey Theory, see the book of Graham,
Rothschild and Spencer \cite{book}.  We make use of van der Waerden's
theorem \cite{vdw}, and certain formulations of our main conjecture
have a similar flavour to the Hales-Jewett theorem \cite{hjt}; for
both of these, see~\cite{book}.  For further related results
and problems, we refer the reader to the original sequence of papers
\cite{egmrss}, \cite{egmrss2} and \cite{egmrss3} by \erdos, Graham,
Montgomery, Rothschild, Spencer and Straus, and to the later survey
\cite{conj} of Graham.  For previous work on subtransitive sets, see
Johnson \cite{johnson}.

Our notation is standard.  In particular, for 
natural numbers $m$, $n$ with $m\le n$
we write
$[n]$ for the set $\{1,2,\ldots\,,n\}$ and $[m,n]$ for the set
$\{m,m+1,\ldots\,n\}$.  For any set $A$, we write $A^{(m)}$ 
for the collection of subsets of
$A$ of order $m$.  If $A$ is a set of integers and $j$ is an integer, we
write $A+j$ for the set $\{i+j:i\in A\}$.  We write $S_n$ to denote the
symmetric group of all $n!$ permutations of $[n]$.
\end{section}

\begin{section}{Is every transitive set Ramsey?}\label{if}
We now discuss the `if' direction of Conjecture \ref{main}:  how might
we prove that every transitive set is Ramsey?

Refining the notion of Ramsey,
we say that a set $Y\subset\R^d$ is {\it $k$-Ramsey for $X$} if any
$k$-colouring of $Y$ yields a subset congruent to $X$.  It follows from
a compactness argument that if $X$ is Ramsey then, for each $k$,
there is a finite set $Y$ such that $Y$ is $k$-Ramsey for $X$ (see
\cite{egmrss} or \cite{book}).  

As we remarked above, every set known to be Ramsey is proved so by
first embedding it in a finite transitive set $X$.  In fact, all such proofs
then continue by colouring a large product $X^n$, or, more precisely,
a scaling $\lambda X^n=\{\lambda x:x\in X^n\}$.  
(For $X\subset\R^m$, we view $X^n$ as a subset of $\R^{mn}$.)
This leads us to make
the following stronger conjecture.

\begin{conjecture}\label{products}
Let $X\subset\R^m$ be a finite transitive set.  
Then, for any $k$, there exists an
$n$ such that some scaling of $X^n$ is $k$-Ramsey for $X$.
\end{conjecture}

The `if' direction of Conjecture \ref{main} would, of course, follow
immediately from Conjecture \ref{products}.  

We digress for a moment to comment on a related notion.  
In their initial paper \cite{egmrss} on Euclidean Ramsey
theory, \erdos, Graham, Montgomery, Rothschild, Spencer and Straus
define a set $X$ to be {\it sphere-Ramsey} if for any positive 
integer $k$ there exists a positive integer $d$ and a positive real number
$r$ such that whenever $\{x\in\R^{d+1}:\|x\|=r\}$,
the $d$-dimensional sphere of
radius $r$, is $k$-coloured it
contains a subset congruent to $X$.  It is obvious that any sphere-Ramsey
set must be Ramsey.  As we observed earlier, it is easy to show that
any transitive set is spherical.  Hence Conjectures \ref{main} and
\ref{products} would together imply that a set is sphere-Ramsey if and only
if it is Ramsey.

Our aim in the remainder of this section will be to reformulate Conjecture
\ref{products} as an equivalent purely combinatorial statement (Conjecture 
\ref{patterns}).
Our first step is to
`remove the geometry' by recasting Conjecture \ref{products}
in terms of the symmetry group of $X$.

Let $G$ be a group, $n$ a positive integer and $I\subset[n]$.  Suppose
$\gv=(g_1,g_2,\ldots\,,g_n)\in G^n$ and $h\in G$.  We write $\gih$ for the
word $(k_1,k_2,\ldots\,,k_n)\in G^n$ where
$$k_i=\left\{\begin{array}{cl}
g_i&\hbox{ if }i\not\in I\\
g_ih&\hbox{ if }i\in I
\end{array}\right..$$
As we shall see, Conjecture~\ref{products} is equivalent to the following
conjecture.

\begin{conjecture}\label{groups}
Let $G$ be a finite group.  Then for any positive integer $k$ there exist
positive integers $n$ and $d$ such that whenever $G^n$ is $k$-coloured
there exists a word $\gv\in G^n$ and a set $I\subset[n]$ with $|I|=d$
such that the set $\{\gih:h\in G\}$ is \mono.
\end{conjecture}

The reader familiar with the Hales-Jewett theorem~\cite{hjt} will see some
resemblance.   Indeed, without the restriction `$|I|=d$', 
Conjecture~\ref{groups}
would be an easy consequence of the Hales-Jewett theorem.

We remark that in Conjecture~\ref{groups} we cannot insist that $d=1$.
More generally, we cannot even insist that $\gv$ be constant on $I$.  Indeed,
to see this,  $2$-colour $G^n$ by colouring a vector according to
whether the number of occurrences of the identity element $e$ lies between
$1$ and $d$ or between $d+1$ and $2d$ (modulo $2d$).

We next show how Conjecture~\ref{products} may be deduced directly from
Conjecture~\ref{groups}.

\begin{prop}
Conjecture \ref{groups} implies Conjecture \ref{products}
\end{prop}

\begin{proof}
Assume that Conjecture \ref{groups} is true.  Let $X$ be a finite
transitive set and let $k$ be a positive integer.  Let $G$ be the symmetry
group of $X$.  By Conjecture \ref{groups}, we may choose $n$ in such a way
that 
whenever $G^n$ is $k$-coloured
there exists a word $\gv\in G^n$ and a set $I\subset[n]$ with $|I|=d$
such that the set $\{\gih:h\in G\}$ is \mono.

Suppose $X^n$ is $k$-coloured.
We induce a $k$-colouring of $G^n$ by picking $x\in X$ and giving $(g_1,g_2,
\ldots\,,g_n)$ the colour of $(g_1(x),g_2(x),\ldots\,,g_n(x))$.  Now choose
$I\subset[n]$ with $|I|=d$ and $\gv\in G^n$ such that the set
$\{\gih:h\in G\}$ is \mono.  For notational convenience, assume \Wlog\ that
$I=[d]$, so that the set
$$Y=\{g_1h(x),\ldots\,,g_dh(x),g_{d+1}(x),\ldots\,,g_n(x):h\in G\}$$
is \mono.  

As $G$ acts transitively on $X$, we have that
$X=\{h(x):h\in G\}$, and so
$$Z=\{(\underbrace{h(x),\ldots\,,h(x)}_d,\underbrace{x,\ldots\,,x}_{n-d}):
h\in G\}$$ is a scaling of $X$ by factor $\sqrt{d}$.  But $g_1$, $\ldots\,$,
$g_m$ are isometries and so $Y$ and $Z$ are isometric.  Hence 
${1\over\sqrt{m}}X^n$
is $k$-Ramsey for $X$.
\end{proof}

Conjectures~\ref{products} and~\ref{groups} are in fact equivalent, but
we have no simple and direct way of deducing Conjecture~\ref{groups} from
Conjecture~\ref{products}.
In the case of a group $G$ which acts as the symmetry group of
some transitive set of order $|G|$, Conjecture
\ref{groups} may be deduced easily from Conjecture \ref{products}.
But, for example, the cyclic group $C_3$ does not act as the symmetry
group of any
$3$-point set.  The proof of this direction of the equivalence must wait
until the end of the current section.

Having successfully removed the geometry, our next task is to
`remove the groups'. One
may think of Conjecture~\ref{groups} as saying that the varying parts of
our $m$ words ($m=|G|$) contain the group table of $G$, possibly with
some columns omitted and some columns repeated.  That is, we obtain the
$m$ rows of some Latin-square-type pattern.  Instead of asking merely
for these $m$ rows, we could instead demand all $m!$ permutations of the
elements of $G$.  This now gives us a purely combinatorial 
statement---a `fixed-block-size'
Hales-Jewett-type conjecture.

We first need a preliminary definition.  
We wish to consider a collection of words of the following form:
we fix $m$ blocks, make our words the same as each other outside
these blocks, and take each of the $m!$ possible arrangements of
$1$, $2$, $\ldots\,$, $m$ amongst the blocks.  Formally,
a {\it block permutation set}
in $[m]^n$ is a set $B$ formed in the following way.  First, select
pairwise
disjoint subsets $I_1$, $\ldots\,$, $I_m\subset[n]$ 
and elements $a_i\in[m]$ for each 
$i\not\in\bigcup_{j=1}^mI_j$. 
For each $\pi\in S_m$, define $a^\pi\in[m]^n$ by
$$(a^\pi)_i=\left\{\begin{array}{cl}
\pi(j)&\hbox{if }i\in I_j\\
a_i&\hbox{if }i\not\in\bigcup_{j=1}^mI_j
\end{array}\right..$$
Now set $B=\{a^\pi:\pi\in S_m\}\subset[m]^n$.

If $\sum_{j=1}^m|I_j|=d$ then we say that $B$ is of {\it degree} $d$.
We sometimes refer to the sets $I_1$, $I_2$, $\ldots\,$, $I_m$ as
{\it blocks.}

We remark that a block permutation set need not contain precisely $m!$ 
elements:  some of the blocks could be empty.

\begin{conjecture}\label{hj}
Let $m$ and $k$ be positive integers.  Then there exist positive integers
$n$ and $d$ such that whenever $[m]^n$ is $k$-coloured it contains a
block permutation set of degree $d$.
\end{conjecture}

Again, note that if the degree condition were omitted then this would
follow easily from the Hales-Jewett theorem.  Indeed, any
block permutation set in $[m]^n$ is contained in some $m$-dimensional
combinatorial subspace ($m$-parameter set).
Furthermore, as with Conjecture~\ref{groups}, we cannot require each
block to have size $1$.

It is clear that Conjecture \ref{groups} follows from Conjecture
\ref{hj}.  While Conjecture \ref{hj} appears much stronger, it is in
fact equivalent to Conjecture \ref{groups}, as we now show.  (In fact,
this follows from later results; we include it here because the proof
is concise and direct.)

\begin{prop}
Conjectures \ref{groups} and \ref{hj} are equivalent.
\end{prop}

\begin{proof}
For the non-trivial direction of the implication,
assume Conjecture \ref{groups} is true, and let $m$ and $k$ be positive
integers.  We apply Conjecture \ref{groups} to the symmetric group $S_m$
and obtain integers $n$ and $d$ as above.  Now suppose that $[m]^n$ is
$k$-coloured.  We induce a colouring of $S_m^n$ by giving
$(\pi_1,\ldots\,,\pi_n)\in S_m^n$ the colour of 
$(\pi_1^{-1}(1),\ldots\,,\pi_n^{-1}(1))\in[m]^n$.  Now choose
$\piv=(\pi_1,\ldots\,,\pi_n)\in S_m^n$ and $I\subset[n]$ with $|I|=d$
such that the set $\{\pis:\sigma\in S_m\}$ is \mono.  For simplicity
of notation, assume \Wlog\ that $I=[d]$.  Thus the set
$$\{(\sigma^{-1}\pi_1^{-1}(1),\ldots\,,\sigma^{-1}\pi_d^{-1}(1),
\pi_{d+1}^{-1}(1),\ldots\,,\pi_n^{-1}(1)):\sigma\in S_m\}\subset[m]^n$$
is \mono.  We now take $I_j=\{i\in I:\pi_i^{-1}(1)=j\}$ for each $j=1$, $2$,
$\ldots\,$, $m$.
\end{proof}

We now give a further equivalent formulation of this 
conjecture which we hope
might be more amenable to proof.

Conjecture \ref{hj} asks for a collection of words of a certain type all of 
the same
colour.    We may think of this collection as
being represented by the `pattern' $12\ldots m$.  More generally, we could
consider an arbitrary pattern.  For example, to realise the pattern
$11223$ we would demand five blocks, and a word for each way of assigning
the symbol $1$ to two of the blocks, the symbol $2$ to another 
two of the blocks and the symbol $3$ to
the remaining block.  So in this case we would have a total of 30 words (if
all of the blocks were non-empty).

Formally, we define a {\it template} over $[m]$ to be a non-decreasing
word $\tau\in[m]^\ell$ for some $\ell$. 

Next, we define a {\it block set with template $\tau$.}  The reader should bear
in mind our earlier definition of a block permutation set, which is a certain
special case:  it is a block set with template $12\ldots m$.

Suppose that $\tau\in[m]^\ell$ is a template.
For each $j\in[m]$, let $c_j$ be the number of times that the symbol
$j$ appears in the template $\tau$; that is,
$c_j=|\{i\in[\ell]:\tau_i=j\}|$.  Note that
$\sum_{j=1}^mc_j=\ell$.
We define the set $S$ of {\it rearrangements} of $\tau$ by
$$S=\Big\{\,\pi\in[m]^\ell:\big|\{i\in[\ell]:
\pi_i=j\}\big|=c_j\quad\forall j\in[m]
\,\Big\}\,.$$
A {\it block set with template $\tau$}
in $[m]^n$ is a set $B$ formed in the following way.  First, select
pairwise
disjoint subsets $I_1$, $\ldots\,$, $I_\ell\subset[n]$ 
and elements $a_i\in[m]$ for each 
$i\not\in\bigcup_{j=1}^\ell I_j$. 
For each $\pi\in S$, define $a^\pi\in[m]^n$ by
$$(a^\pi)_i=\left\{\begin{array}{cl}
\pi_j&\hbox{if }i\in I_j\\
a_i&\hbox{if }i\not\in\bigcup_{j=1}^\ell I_j
\end{array}\right..$$
Now set $B=\{a^\pi:\pi\in S\}\subset[m]^n$.

As before, if $\sum_{j=1}^\ell|I_j|=d$ then we say that 
$B$ is of {\it degree} $d$, and we may
refer to the sets $I_1$, $I_2$, $\ldots\,$, $I_\ell$ as
{\it blocks.}  

We are now ready to state the conjecture.

\begin{conjecture}\label{patterns}
Let $m$ and $k$ be positive integers and let $\tau$ be a template over $[m]$.
Then there exist positive integers $n$ and $d$ such that whenever $[m]^n$
is $k$-coloured it contains a \mono\ block set of degree $d$ with
template $\tau$.
\end{conjecture}

It is easy to see that Conjectures \ref{hj} and \ref{patterns} are
equivalent.  Conjecture \ref{hj} is simply the special case of
Conjecture \ref{patterns} for the template $12\ldots m$.  In the other
direction, Conjecture \ref{patterns} for the template $\tau_1\ldots
\tau_\ell$ on alphabet $[m]$ follows immediately from Conjecture
\ref{hj} on alphabet $[\ell]$. Indeed, given a colouring of $[m]^n$ we
use the map $(x_1,\ldots\,,x_n)\mapsto (\tau_{x_1},\ldots,\tau_{x_n})$
to induce a colouring of $[\ell]^n$ and then apply
Conjecture~\ref{hj}.

What can we say about the initial cases of Conjecture~\ref{patterns}?  The
case $m=1$ is trivial.  In the next case, $m=2$, the conjecture is
true for all templates by an easy application of Ramsey's theorem,
as we shall explain in \S\ref{cases}.  We shall then prove the first 
non-trivial cases:  templates of the form $1\ldots12\ldots23$.

It turns out that, in all of these cases, the 
block set we produce is 
{\it uniform:} that is, all of its blocks have the same size.  
This suggests the following 
conjecture, which perhaps appears more natural.

\begin{conjecture}\label{uniform}
Let $m$ and $k$ be positive integers and let $\tau$ be a template over $[m]$.
Then there exist positive integers $n$ and $d$ such that whenever $[m]^n$
is $k$-coloured it contains a \mono\ {\rm uniform}
block set of degree $d$ with template $\tau$.
\end{conjecture}

At first sight, Conjecture~\ref{uniform} may appear rather stronger than
Conjecture~\ref{patterns}---we have imposed an additional constraint on
the monochromatic block set that we seek.  However, they do in fact
turn out to be equivalent.  We do not see how to deduce 
Conjecture~\ref{uniform} directly from Conjecture~\ref{patterns}, but instead
go via the geometric Conjecture~\ref{products}.


To prove that all of the conjectures of this section 
(Conjectures \ref{products}--\ref{uniform}) are equivalent, it now suffices to show
that we may deduce Conjecture~\ref{uniform} from Conjecture~\ref{products}.
So, how might we
use Conjecture~\ref{products} to deduce
Conjecture~\ref{uniform} for the template $1\ldots m$? One approach
would be to `embed' the template into $\R^m$ as follows.

Suppose $\alpha_1,\alpha_2,\ldots,\alpha_{m}$ are real numbers. Let $X$
be the set in $\R^{m}$ of all permutations of the vector
$(\alpha_1,\alpha_2,\ldots,\alpha_{m})$ and
let $Y=\{\alpha_1,\alpha_2,\ldots,\alpha_m\}^{nm}$. Note that $
X^n\subset Y$. We
shall think of $Y$ as the image of $[m]^{nm}$ in the obvious way.

Certainly $X$ is transitive ($S_m$ acts on it), so by
Conjecture~\ref{products} we know that (provided $n$ is sufficiently
large) whenever $Y$ is $k$-coloured it contains a monochromatic set
congruent to a fixed scaling $sX$ of $X$. One way such a set could
occur is as the image of an $s^2$-uniform block set, but of
course there may be many other ways as well.

The heart of the proof is to ensure that the \emph{only} subsets of
$Y$ congruent to $sX$ are the images of $s^2$-uniform block sets.

\begin{lemma}\label{2.3}
  Let $\alpha_1,\alpha_2,\ldots,\alpha_{m}$ be algebraically
  independent real numbers and define $X$ and $Y$ as above. Then, for
  $s>0$, every subset of $Y$ congruent to $sX$ is the image of an
  $s^2$-uniform block set.
\end{lemma}
\begin{proof} 
Let $W\subset Y$ be congruent to $sX$.  As there is an isometry from $sX$ to 
$W$, we may find a bijection $\theta\colon X\to W$ with 
$\|\theta(x)-\theta(x')\|=s\|x-x'\|$ for all $x$, $x'\in X$.  
For each
  permutation $\pi$, let 
$x_\pi=(\alpha_{\pi(1)},\alpha_{\pi(2)},\ldots,\alpha_{\pi(m)})$ and let
$y_\pi=\theta(x_\pi)$.  So $W=\{y_\pi:\pi\in S_m\}$.  We denote the identity
permutation by $\id$.

We first consider the points $y_\id$ and $y_{(12)}$. Since
$\|x_\id-x_{(12)}\|^2=2(\alpha_2-\alpha_1)^2$ we have that
$\|y_\id-y_{(12)}\|^2=2s^2(\alpha_2-\alpha_1)^2$. Similarly,
$\|y_\id-y_{(13)}\|^2=2s^2(\alpha_3-\alpha_1)^2$. 

For distinct $i$, $j\in[mn]$, let $\lambda_{ij}$ denote the number of
coordinates in which $y_\id$ takes value $\alpha_i$ but $y_{(12)}$
takes value $\alpha_j$.  Similarly, let $\mu_{ij}$ denote the number of
coordinates in which $y_\id$ takes value $\alpha_i$ but $y_{(13)}$
takes value $\alpha_j$.  Then
\[
\|y_\id-y_{(12)}\|^2=\sum_{i,j}\lambda_{ij}(\alpha_i-\alpha_j)^2\qquad\textrm{and}\qquad
\|y_\id-y_{(13)}\|^2=\sum_{i,j}\mu_{ij}(\alpha_i-\alpha_j)^2
\]
and so
\[
(\alpha_3-\alpha_1)^2\sum_{i,j}\lambda_{ij}(\alpha_i-\alpha_j)^2=(\alpha_2-\alpha_1)^2\sum_{i,j}\mu_{ij}(\alpha_i-\alpha_j)^2.
\]

By comparing the coefficients of appropriate monomials in $\alpha_1$,
$\ldots\,$, $\alpha_m$, it is easy to check that all of the $\lambda_{ij}$ and
$\mu_{ij}$ must be zero, except for $\lambda_{12}$ and $\mu_{13}$ which must be
equal. (For example considering the coefficient of $\alpha_3^4$ and
using the fact that $\lambda_{ij}\ge 0$ for all $i,j$ reveals that
$\lambda_{ij}=0$ if either $i$ or $j$ is 3; then considering
the coefficient of
$\alpha_3^2\alpha_i\alpha_j$ gives that $\lambda_{ij}=0$ unless $(i,j)=(1,2)$.)
In other words, this means that when changing from $y_\id$ to
$y_{(12)}$, all that happens is that some coordinates change from $\alpha_1$ to
$\alpha_2$ and some other coordinates change from $\alpha_2$ to
$\alpha_1$
Moreover, as
$\lambda_{12}=2s^2$, we see that the total number of 
coordinates
that change is $2s^2$.

\newcommand \ai{\alpha_{\pi(i)}}
\newcommand \aj{\alpha_{\pi(j)}}
\newcommand \ak{\alpha_{\pi(k)}}

More generally, the same argument shows that, for any permutation $\pi$
and transposition $(ij)$, to change from $y_\pi$ to $y_{\pi(ij)}$ it is
only necessary to change some coordinates from $\ai$ to $\aj$ and some
other coordinates from $\aj$ to $\ai$.
Let $U_\pi(ij)$ be the set of coordinates which change from
$\ai$ to $\aj$ when one changes $y_\pi$ into $y_{\pi(ij)}$.
Again as above, we have
$|U_\pi(ij)|+|U_\pi(ji)|=2s^2$
for all $\pi$, $i$ and $j$. As the set $U_\pi(ij)$ is a subset of the
coordinates in $y_\pi$ which equal $\ai$, we have that, for fixed
$\pi$ and $j$, the sets $U_\pi(ij)$ are pairwise disjoint as $i$
varies.  Moreover, for distinct $i$ and $j$ the sets $U_\pi(ij)$ and
$U_\pi(ji)$ are disjoint.  Of course, for fixed $\pi$ and $i$, the
sets $U_\pi(ji)$ could intersect (and, indeed, we shall show that not
only do they intersect but, in fact, they are identical).

Now, consider the distance $\|y_{(ij)\pi}-y_{(ik)\pi}\|^2$.  
The vectors $x_{\pi(ij)}$ and
$x_{\pi(ik)}$ are equal in all coordinates except for coordinates $i$, $j$
and $k$.  In these coordinates the vector $x_{\pi(ij)}$ takes values
$\aj$, $\ai$ and $\ak$ respectively; while the vector $x_{\pi(jk)}$ takes the
values $\ak$, $\aj$, and $\ai$ respectively.
Hence
\[
\|x_{\pi(ij)}-x_{\pi(ik)}\|^2=\left((\ai-\aj)^2+(\aj-\ak)^2+(\ak-\ai)^2\right)
\]
and so 
\[
  \|y_{\pi(ij)}-y_{\pi(ik)}\|^2=s^2\left((\ai-\aj)^2+(\aj-\ak)^2+(\ak-\ai)^2\right).
\]

An alternative way to calculate this distance is to consider explicitly how
the vector $y_{\pi(ij)}$ differs from the vector
$y_{\pi(ik)}$.  Imagine that we first change $y_{\pi(ij)}$ to $y_\pi$ and then
to $y_\pi(jk)$.  
When going from
$y_{\pi(ij)}$ to $y_\pi$ the coordinates in $U_\pi(ij)$ change from $\aj$
to $\ai$ and the coordinates in $U_\pi(ji)$ 
change from $\ai$ to $\aj$. Then when
going from $y_\pi$ to $y_{\pi(ik)}$ the coordinates in $U_\pi(ik)$ change
from $\ai$ to $\ak$ and the coordinates in $U_\pi(ki)$ change from $\ak$ to
$\ai$. Hence
\begin{align*}
\|y_{\pi(ij)}-y_{\pi(ik)}\|^2=|U_\pi(ji)|&(\aj-\ai)^2 + |U_\pi(ki)|(\ak-\ai)^2\\
&+|U_\pi(ij)\setminus U_\pi(ik)|(\aj-\ai)^2\\
&+|U_\pi(ik)\setminus U_\pi(ij)|(\ak-\ai)^2\\
&+|U_\pi(ij)\cap U_\pi(ik)|(\aj-\ak)^2.
\end{align*}

We now have two expressions for $\|y_{\pi(ij)}-y_{\pi(jk)}\|^2$, each of which
is a polynomial in $\alpha_1$, $\ldots\,$, $\alpha_m$ with rational coefficients
(recall that $2s^2$ is an integer).  
Comparing coefficients of $\aj\ak$ gives that
$|U_\pi(ij)\cap U_\pi(ik)|=s^2$.
In particular for every $\pi$ and every $i,j$ we have that $|U_\pi(ij)|\ge
s^2$ and thus, as $|U_\pi(ij)|+|U_\pi(ji)|=2s^2$, that 
$|U_\pi(ij)|=s^2$ for all $\pi$, $i$
and $j$. Therefore the set $U_\pi(ij)$ is
independent of $j$.

Finally, it follows easily from the definition that
$U_\pi(ij)=U_{\pi(ij)}(ij)$. Hence, for any $\ell\not=i$, we have
$U_{\pi(i\ell)}(ij)=U_{\pi(i\ell)}(i\ell)=U_{\pi}(i\ell)=U_\pi(ij)$.  But $S_n$ is
generated by the transpositions of the form $(i\ell)$ for
$\ell\in[m]\setminus\{i\}$, so in fact $U_\pi(ij)$ is independent of
$\pi$ (in addition to being independent of $j$).  We may thus define
$I_i=U_\pi(ij)$.  It is now clear that $\{y_\pi:\pi\in S_k\}$ is
exactly the image of an $s^2$-uniform block set with blocks $I_1$,
$I_2$, $\ldots\,$, $I_k$.
\end{proof}
\begin{prop}
Conjecture~\ref{products} implies Conjecture~\ref{uniform}.
\end{prop}
\begin{proof}
  Suppose Conjecture~\ref{products} holds. We shall deduce 
Conjecture~\ref{uniform} for the
  template $1\ldots m$: the full Conjecture~\ref{uniform} 
then follows exactly as
  Conjecture~\ref{patterns} follows from Conjecture~\ref{hj}.

  Fix $k$ and let $\alpha_1,\ldots,\alpha_m$ be algebraically
  independent real numbers. Form the set $X$ as in Lemma~\ref{2.3}. By
  Conjecture~B there are $n$ and $s$ such that $X^n$
  is $k$-Ramsey for $sX$. Let $Y$ be as defined in Lemma~\ref{2.3}.
  
  Suppose that $[m]^{mn}$ is $k$-coloured. This induces a colouring of
  $Y$, and $Y$ contains $X^n$. Hence $Y$ contains a monochromatic copy
  of $sX$, and by Lemma~\ref{2.3} this is exactly the image of an
  $s^2$-uniform block set with template $1\ldots m$.
\end{proof}

\end{section}
\begin{section}{The first cases of Conjecture~\ref{patterns}}\label{cases}
In this section, we consider some small cases of Conjecture \ref{patterns}.  
As we mentioned in \S\ref{if}, the case $m=1$ is trivial.  In the case 
$m=2$ the conjecture is
true for all templates by an easy application of
Ramsey's theorem, as we now explain.

Consider the template $\underbrace{1\ldots1}_r\underbrace{2\ldots2}_s$ and
let $k\in\N$.  
By Ramsey's theorem, there exists a positive integer $n$ such that
whenever $[n]^{(s)}$ is $k$-coloured it has a monochromatic subset of
size $r+s$.  Suppose now $[2]^n$ is $k$-coloured.  We induce a
$k$-colouring of $[n]^{(s)}$ by giving $A\in[n]^{(s)}$ the colour of the word
$a^A\in[2]^n$ defined by
$$(a^A)_i=\left\{\begin{array}{cl}
1&\hbox{if }i\not\in A\\
2&\hbox{if }i\in A
\end{array}\right..$$
Let $B\in[n]^{(r+s)}$ be monochromatic.  Now take $I_1$, $I_2$, $\ldots\,$,
$I_{r+s}$ to be the singleton subsets of $B$ and $a_i=1$ for all
$i\not\in B$, giving our \mono\ block set.

We now prove the first non-trivial cases of Conjecture~\ref{patterns}.
In a sense, the first non-trivial case corresponds to the template
$123$.  In fact, we prove the stronger result that the conjecture holds
for all
templates of the form $1\ldots12\ldots23$.  Note that in what follows,
the proof can be simplified for templates with only one 1, i.\,e.~those
of the form $12\ldots23$:  in this case, the application of van der
Waerden's theorem is replaced by the pigeonhole principle.

\begin{theorem}\label{00112}
Conjecture \ref{patterns} is true for $m=3$ and templates of the form
$\underbrace{1\ldots1}_{r}\underbrace{2\ldots2}_{s}3$.
\end{theorem}

\begin{proof}
For the sake of definiteness, we begin by fixing the values of certain
parameters.

Let $\ell=r+s+1$.

By van der Waerden's theorem, there exists a positive integer
$a$ such that whenever $[0,a-1]$ is $k$-coloured there exists
a \mono\ \ap\ or length $r+1$.  Let $t=a!$ and $d=\ell t$.

By Ramsey's theorem, there exists a positive integer $b$ such that
whenever $[b]^{(t)}$ 
is $k^a$-coloured, there exists a \mono\ subset
of order $d$.  Let $u=b+a-1$ and $v=u+tr$

By Ramsey's theorem again, there exists a positive integer $n$ such that
whenever $[n]^{(u)}$ is $k^{u\choose t}$-coloured there exists a \mono\ 
subset of order $v$.

Now suppose $[3]^n$ is $k$-coloured, say by $c_1\colon[3]^n\to[k]$.
We shall consider only the set $A\subset[3]^n$ of those words
containing precisely $t$ 3's and $u-t$ 2's.

Let $B\subset[2]^n$ be the set of words of length $n$ containing
$u$ 2's and $n-u$ 1's, and let $C\subset\{2,3\}^u$ be the collection of
words of length $u$ containing $t$ 3's and $u-t$ 2's.  There is an obvious
bijection $\theta\colon B\times C\to A$:  define $\theta(w,x)$ to be the word
obtained by replacing the $u$ 2's in $w$ by the letters of the word $x$
(in order).  We may thus induce a $k^{u\choose t}$-colouring $c_2$ of $B$
by the complete colouring of $\{(w,x):x\in C\}$.  Formally, let
$C=\{x_1,\ldots\,,x_{u\choose t}\}$, 
and define $c_2\colon B\to[k]^{u\choose t}$
by $c_2(w)=(c_1(\theta(w,x_1)),\ldots\,,c_1(\theta(w,x_{u\choose t})))$.

Similarly to the case $m=2$, this yields a ${k^{u\choose t}}$-colouring
$c_3$ of $[n]^{(u)}$: we give a set $X\in[n]^{(u)}$ the colour of the word
$w^X\in B\subset[2]^n$ with $$(w^X)_i=\left\{\begin{array}{cl}
1&\hbox{if }i\not\in X\\
2&\hbox{if }i\in X
\end{array}\right..$$
By definition of $n$, there is a $c_3$-\mono\ subset $D\subset[n]$ of order
$v$.  For notational simplicity, we assume \Wlog\ that $D=[v]$.  

What we have
proved is that for words $w\in A$
where all 2's and 3's are contained within the first
$v$ positions, the colour $c_1(w)$ depends only on the relative ordering
of the 2's and the 3's: `the  positions of the 1's do not matter'.
Thus we may induce a $k$-colouring $c_4$ of the subset $E\subset\{2,3\}^u$
of words with precisely $t$ 3's.  Formally, we define
$c_4(w)=c_1(w')$ where $w'$ is the word consisting of $w$ followed by
$n-u$ 1's.

We next induce a $k^a$-colouring $c_5$ of $[b]^{(t)}$ by colouring the
set $R\in[b]^{(t)}$ according to the colours of the following $a$ words:
the word the positions of whose 2's form the set $R$,
the word the positions of whose 2's form the set $R+1$,
the word the positions of whose 2's form the set $R+2$, and so on.
That is, we define $c_5\colon[b]^{(t)}
\to [k]^a$ by $c_5(R)=(c_4(w^{R,0}),\ldots\,,c_4(w^{R,a-1}))$
where $w^{R,j}\in E$ is defined by
$$(w^{R,j})_i=\left\{\begin{array}{cl}
2&\hbox{if }i-j\not\in R\\
3&\hbox{if }i-j\in R
\end{array}\right..$$
By definition of $b$, we may extract a $c_5$-\mono\ subset
$F\subset[b]$ of order $d$.

This now induces a $k$-colouring $c_6$ of $[0,a-1]$ by
$c_6(j)=c_5(R+j)$ where $R\in F^{(t)}$.  
(Note that this does not depend on the choice of $R$.)
So by definition of $a$, there is a \mono\ \ap\ of length $r+1$,
say $p-rq$, $p-(r-1)q$, $p-(r-2)q$,
 $\ldots\,$, $p$.

Write $F=\{i_1,i_2,\ldots\,,i_d\}$ with $i_1<i_2<\cdots<i_d$.  For
$1\le j\le\ell$, let
$$I_j=\{p+i_{(\lambda\ell+j-1)q+\mu}+rq\lambda:
0\le\lambda<{t\over q},1\le\mu\le q\}.$$
Since $t=a!$ and $q<a$, we must have that $q$ is a factor of $t$,
and so $|I_j|=t$ for each $j$.  Furthermore, the largest element appearing
in any of the $I_j$ is $p+i_d+r(t-q)<a-1+b+rt=v$, so each $I_j\subset[v]$.

For $i\not\in\bigcup_{j=1}^\ell I_j$, let
$$a_i=\left\{\begin{array}{cl}
2&\hbox{if }i\le v\\
1&\hbox{if }i>v
\end{array}\right..$$
Let $S$ be the set of rearrangements of the template
$\underbrace{1\ldots1}_{r}\underbrace{2\ldots2}_{s}3$.
It remains to check that for each $\pi\in S$ the words $a^\pi$ all have the
same colour.

So let $\pi\in S$.  It is clear that $a^\pi$ contains $t$ 3's and
$rt$ 1's, and so
$v-tr-t=u-t$ 2's. 
Moreover, as each $I_j\subset[v]$, 
all of the 2's and 3's in $a^\pi$ are contained
within the first $v$ positions of $a^\pi$.  Hence the colour of $a^\pi$ is
determined completely by the relative positions of the $2$'s and 3's.

Write $I$ for the set of positions of the 3's amongst the 2's and 3's in 
$a^\pi$ (so $I\in[u]^{(t)}$).  As our template contains only one 3, 
there is a unique $j\in[\ell]$
such that $\pi_j=3$.  Let $h$ be the number of 1's appearing to the left
of the unique 3 in $\pi$.  Then
\begin{eqnarray*}
I&=&\{p+i_{(\lambda\ell+j-1)q+\mu}+rq\lambda-(r\lambda+h)q:
0\le\lambda<{t\over q},1\le\mu\le q\}\\
&=&\{p+i_{(\lambda\ell+j-1)q+\mu}-hq:0\le\lambda<{t\over q},
1\le\mu\le q\}\\
&=&X_j+p-hq,
\end{eqnarray*}
where 
$$X_j=\{i_{(\lambda\ell+j-1)q+\mu}:0\le\lambda<{t\over q},1\le\mu\le q\}
\in F^{(t)}$$
and $0\le h\le r$.
The result follows.
\end{proof}

We have now proved Conjecture \ref{patterns} for templates of the form
$1\ldots12\ldots23$.  So the first open case is:

\begin{problem}\label{001122}
Prove Conjecture \ref{patterns} with $m=3$ for the template $112233$.
\end{problem}

Finally, we remark that the uniformity of the block sets in our proof
of Theorem~\ref{00112} is already enough to yield some new examples of
Ramsey sets.  Indeed, from the template
$\underbrace{1\ldots1}_{r}\underbrace{2\ldots2}_{s}3$ we obtain that,
for any distinct reals $\alpha$, $\beta$ and $\gamma$, the set
$X\subset\R^{r+s+1}$ consisting of all those points $x$ having $r$
coordinates $\alpha$, $s$ coordinates $\beta$ and one coordinate
$\gamma$ is Ramsey.  In general, the set $X$ does not satisfy the
conditions of \kriz's theorem~\cite{kriz}.  However, we do not know
whether or not it embeds into a larger set that does.
\end{section}

\begin{section}{Not all spherical sets are \fp}\label{gons}
In this section, we show that our conjecture is genuinely different from
the old conjecture~\cite{conj}, by showing that there exists a finite spherical
set that is not \fp.  Specifically, we show that if $k\geqslant16$ then there
exists a cyclic $k$-gon that is not \fp.

One natural approach is to aim for a non-constructive proof showing that
almost no cyclic $k$-gon is subtransitive.  However, the space of cyclic
$k$-gons has $k$ degrees of freedom, whereas the space of (non-isometric)
orbits of a fixed group of isometries of $\R^n$ can have many more.  This
is the main obstacle that we have to overcome.

For convenience, we consider labelled, oriented
$k$-gons, i.\,e.~we label the vertices $1$, $2$, $\ldots\,$, $k$ in clockwise
order and consider two the same if there is an isometry between them
preserving the labels.  Now suppose that $x_1\ldots x_k$ is a cyclic
$k$-gon with circumcentre $x_0$ and circumradius $r$.  Then it is uniquely
determined by the ordered $k$-tuple $(r,\angle x_1x_0x_2, \angle x_1x_0x_3,
\ldots\,,\angle x_1x_0x_k)$.
This allows us to think of the
set of cyclic $k$-gons as a subset $\p\subset\R^k$ of non-zero 
($k$-dimensional Lebesgue) measure.

We show that, for any $k\geqslant16$, the set of \fp\ cyclic $k$-gons 
has measure zero.  
The reader is cautioned that, throughout what follows, when we refer to
`orthogonal planes' we use the term in the sense of orthogonal affine
subspaces of a real vector space.  That is, two planes $\Pi_1$ and $\Pi_2$
are orthogonal if for all $x_1$, $y_1\in\Pi_1$ and $x_2$, $y_2\in\Pi_2$
we have $(x_1-y_1)\perp(x_2-y_2)$.  (So, in particular, it is not possible
to find two orthogonal planes in $\R^3$.)

The idea of the proof is as follows.  We show that every transitive
k-gon can be embedded in some $\R^n$ as $g_1(y)\ldots g_n(y)$ for some
$y\in\R^n$ and one of countably many different $k$-tuples
$(g_1,\ldots\,,g_k)$ of orthogonal transformations of $\R^n$.  So we
may assume that $n$ and $(g_1,\ldots\,,g_k)$ are given.  We begin by
fixing a `reference' $k$-gon $g_1(x)\ldots g_k(x)$.  Now, for any
$y\in\R^n$, the distances $\|g_i(x)-g_i(y)\|$ ($1\le i\le k$) are all
the same.  Hence we would like to show that if $x_1\ldots x_k$ is a
$k$-gon then the set of $k$-gons $y_1\ldots y_k$ with all distances
$\|x_i-y_i\|$ the same has dimension strictly less than $k$.
Unfortunately, this is not quite true---there are many ways to
construct such a $y_1\ldots y_k$ in a plane orthogonal to the plane of
$x_1\ldots x_k$.  But this difficulty is easily surmounted: instead of
fixing a single reference $k$-gon, we start from a (necessarily
finite) maximal pairwise-orthogonal family of $k$-gons in $\R^n$.  The
heart of the proof is in the following lemma.

\begin{lemma}\label{distances}
Let $x_1\ldots x_k$ be a fixed cyclic $k$-gon in $\R^n$ with $k\geqslant 16$.  
Let
$\Q\subset\p$ be the set of cyclic $k$-gons which can be embedded in $\R^n$ as
$y_1\ldots y_k$ in such a way that
\newcounter{bean}
\begin{list}{(\roman{bean})}{\usecounter{bean}}
\item $\|x_1-y_1\|=\|x_2-y_2\|=\cdots=\|x_k-y_k\|$; and
\item the planes of $y_1\ldots y_k$ and $x_1\ldots x_k$ are non-orthogonal.
\end{list}
Then $\Q$ has measure zero.
\end{lemma}

\begin{proof}
We may assume \Wlog\ that $n=5$, as any two non-orthogonal planes in
$\R^n$ lie in a $5$-dimensional affine subspace.

We parameterize the space $\p'$ of cyclic $k$-gons in $\R^5$ as follows.
First, choose points $y_1$, $y_2$, $y_3$ in general position in $\R^5$.
These determine a circle $\gamma$, with centre $y_0$, say, and the cyclic
$k$-gon  $y_1\ldots y_k$ is now determined by the angles
$\angle y_1y_0y_4$, $\angle y_1y_0y_5$, $\ldots\,$, $\angle y_1y_0y_k$.  So we
have $\p'\subset\R^{15+(k-3)}=\R^{12+k}$. 

Let $\Q'\subset\p'$ be the set of all possible embeddings $y_1\ldots y_k$
of cyclic
$k$-gons from $\Q$ into $\R^5$ satisfying (i) and (ii).  
Our aim is to show that the dimension
of $\Q'$ is, in fact, much smaller than $12+k$.

Suppose $y_1\ldots y_k\in\Q'$.  Let $r=\|y_1-x_1\|$, let $\gamma$ be
the circle through $y_1$, $y_2$ and $y_3$, and, for each $i\geqslant2$, 
let $S_i$ be the $4$-sphere
with centre $x_i$ and radius $r$.  
For each
$i\geqslant 2$, we must have $y_i\in S_i$.
Moreover, for each $i\geqslant4$ we must also have
$y_i\in\gamma$.  But for each $i$,
either $\gamma\cap S_i$ finite or $\gamma\subset S_i$.

We shall prove that in fact $\gamma\subset S_i$ for at most $2$ distinct
values of $i\geqslant4$.  Assume for a contradiction that $\gamma\subset S_{\ell_1}
\cap S_{\ell_2}\cap S_{\ell_3}$ for some $4\le \ell_1<\ell_2<\ell_3\le k$.
Then for $i$, $j\in\{1,2,3\}$, we have $\|y_j-x_{\ell_i}\|=r$, and so
$y_j\cdot x_{\ell_i}={1\over 2}(\|y_i\|^2+\|x_{\ell_i}\|^2-r)$.  It follows
that if $i_1$, $i_2$, $j_1$, $j_2\in\{1,2,3\}$ 
then $(x_{\ell_{i_1}}-x_{\ell_{i_2}})\cdot(y_{j_1}-y_{j_2})=0$, 
i.\,e.~$x_{\ell_{i_1}}-x_{\ell_{i_2}}$ is perpendicular to
$y_{j_1}-y_{j_2}$.  But this implies that the planes of $x_1\ldots x_k$
and $y_1\ldots y_k$ are orthogonal, a contradiction.  

Hence $\Q'$ is contained in a finite union of $15$-dimensional submanifolds of 
$\R^{k+12}$ ($15=5+4+4+1+1$).
So, for example by Sard's theorem \cite{sard}, $\Q$ has measure zero.
\end{proof}

\begin{theorem}\label{inequiv}
For each $k\geqslant16$, the set of \fp\ cyclic $k$-gons has measure zero.  
In particular,
there exists a cyclic $16$-gon which is not \fp.
\end{theorem}

\begin{proof}
Let ${\mathcal{S}}$ be the set of subtransitive cyclic $k$-gons.
Suppose $P\in{\mathcal{S}}$.  Then $P$ can be embedded
into some $\R^n$ as $y_1\ldots y_k$ in such a way that $y_i=g_i(y)$
($1\le i\le k$) for some $y\in\R^n$ and $g_1$, $g_2$, $\ldots\,$, 
$g_k$ elements of the orthogonal group $O(n)$ with
$\langle g_1,\ldots\,,g_k\rangle$ finite.

Fix such $n$ and $g_1$, $g_2$, $\ldots\,$, $g_k$.  For $y\in\R^n$,
write $\gv(y)$ for the $k$-gon $g_1(y)\ldots g_k(y)$.  Let ${\mathcal{S}}'$ be the
set of $P\in{\mathcal{S}}$ which can be embedded into $\R^n$ as
$\gv(y)=g_1(y)\ldots g_k(y)$.  Let $\gv(x_1)$, $\ldots\,$, $\gv(x_p)$ be
a (necessarily finite) maximal family of pairwise-orthogonal embeddings
of polygons from ${\mathcal{S}}'$ into $\R^n$.  Then if $\gv(y)$ is any embedding of
a polygon from ${\mathcal{S}}'$ into $\R^n$ there must be some $i$ such that
$\gv(y)$ and $\gv(x_i)$ are not orthogonal.  Write $x=x_i$.  As each $g_i$ is
an orthogonal map, we have $$\|g_1(x)-g_1(y)\|=\|g_2(x)-g_2(y)\|=\cdots=
\|g_k(x)-g_k(y)\|.$$ So by Lemma \ref{distances}, ${\mathcal{S}}'$ is a union of
finitely many sets each of measure zero, and hence ${\mathcal{S}}'$ has measure zero.

Now, a finite group $G$ has only countably many orthogonal representations,
up to conjugation by an orthogonal transformation. Indeed, every
orthogonal representation is a direct sum of irreducible representations.
and the group $G$ has only finitely many inequivelant irreducible
linear representations.  Moreover, two
irreducible orthogonal representations which are isomorphic by some linear
map are, in fact, isomorphic by an orthogonal map (because, for example,
by Schur's Lemma there is a unique $G$-invariant inner product on $\R^n$
up to multiplication by a scalar---for more details, see 
e.\,g.~Lemma 4.7.1 of~\cite{wolf}).

Hence the orthogonal groups $O(n)$ have only countably many distinct
finite subgroups (up to conjugation),
and given a finite subgroup of $O(n)$ there are only finitely many ways
to select from it a sequence of $k$ elements.  Thus ${\mathcal{S}}$ is a countable
union of sets of measure zero and so itself has measure zero.
\end{proof}

While our proof shows that almost every cyclic $16$-gon is not \fp, 
it does not provide an explicit construction.
We hope that such an explicit construction of a polygon $P$ might provide
some insight into proving that $P$ is not Ramsey.  We are therefore interested
in a solution to the following problem.

\begin{problem}
Give an explicit construction of a cyclic polygon that is not \fp.
\end{problem}

We also find it unlikely that it is necessary to go as far as $16$-gons:

\begin{conjecture}
There exists a cyclic quadrilateral that is not \fp.
\end{conjecture}

Indeed, we believe that almost no cyclic quadrilateral should be \fp.

We remark that it is easy to check that all trapezia are \fp.  In fact,
\kriz~\cite{trapezia} showed that all trapezia are Ramsey.  This may also be
deduced from his general result in~\cite{kriz}.
\end{section}

\begin{section}{Are Ramsey sets subtransitive?}\label{conclude}
For this direction of our conjecture, we have no results at all---so in
this section we only mention a few heuristic ideas.

Given a Ramsey set $X$, why might there be a transitive set
containing it? There are certainly some `structured' sets containing 
$X$---namely the sets that are $k$-Ramsey (for some $k$) for $X$. Of course, it
is impossible that every such set is transitive, as we may always add points 
to a set that is $k$-Ramsey for $X$ to destroy any symmetry that is
present. So one should focus on the {\it minimal} $k$-Ramsey sets for $X$. 

To fix our ideas, let us consider the simplest possible case, when $X$ is
the set $\{ 0,1 \}$. What are the minimal $2$-Ramsey sets for $X$? For a 
finite set $S$ in $\R^n$, define the {\it graph} of $S$ to be the graph on
vertex-set $S$ in which we join two points if they are at unit distance.
Then the minimal $2$-Ramsey sets for $X$ are precisely those sets whose
graph is an odd cycle. Now, such a set can be very far from transitive: indeed,
it might have no isometries at all. However, there are two key points. The 
first is that, for any such set, we can transform it, preserving unit 
distances, to obtain a transitive set. The second, perhaps more important, is
that a {\it minimum-sized} such set has to be transitive, as it has to be
an equilateral triangle. And similarly for the sets that are $k$-Ramsey for
$X$: minimal such sets are sets whose graphs are $(k+1)$-critical (meaning that
they have chromatic number $k+1$ but the removal of any vertex decreases
the chromatic number), and the unique minimum-sized such set is the
regular simplex on $k+1$ vertices.

Similar phenomena seem to be present in other examples. We cannot
expect in general to focus only on sets that are $2$-Ramsey for a given set 
$X$:  it is certainly possible to find $X$ that is not Ramsey, but such that
there does exist a set $Y$ that is 2-Ramsey for $X$ 
(for example, $X=\{ 0,1,2 \}$---see~\cite{egmrss}). 
But this is perhaps not surprising, as of prime
importance will be how the copies of $X$ `fit together' inside the set $S$,
and one may need more colours to `encode' this information. For example,
it may be that one should look at the $k$-Ramsey sets for $X$, where $k$
is $2^{| X |}$---the intuitive idea being that the colouring of $S$ 
obtained by, for a point $s$ of 
$S$, listing those points of $X$ can map to $s$ in an embedding of $X$ into
$S$, might give key information about the structure of $S$. At any rate, we
wonder if the following is true.

\begin{problem}
Let $X$ be a Ramsey set. Must there exist a $k$ such
that every minimum-sized set that is $k$-Ramsey for $X$ is transitive?
\end{problem}

As an alternative, we suggest the following modification of this
question which is perhaps more approachable.  In place of a
minimum-sized set that is $k$-Ramsey for $X$, we might instead
consider a minimal set $k$-Ramsey for $X$ with the further property
that it cannot be transformed to contain additional copies of $X$
whilst retaining all those already present.

Finally, we consider an algorithmic question.  It is easy to see that one
can determine in finite time whether or not a given set is spherical.
However, it is not clear that it is possible to determine in finite
time whether or not a given set is subtransitive.  So
if our conjecture were true, it would leave open the problem of finding 
an algorithm to determine whether or not a given set is Ramsey.  So we ask:

\begin{problem}
Is there an algorithm for testing in finite time whether or not a given
set is \fp?
\end{problem}

\end{section}


\begin{thebibliography}{99}

\bibitem[1]{cantwell}
Cantwell, K., All regular polytopes are Ramsey.
{\it J.~Comb.~Theory~Ser.~A} {\bf114} (2007), 555--562.

\bibitem[2]{egmrss}
Erd\H os, P., Graham, R.~L., Montgomery, P.,  Rothschild, B.~L., 
Spencer, J., and Straus, E.~G., Euclidean Ramsey Theorems. I., {\it 
J.~Comb.~Theory~Ser.~A} {\bf14} (1973), 341--363.

\bibitem[3]{egmrss2}
Erd\H os, P., Graham, R.~L., Montgomery, P.,  Rothschild, B.~L., 
Spencer, J., and Straus, E.~G., Euclidean Ramsey Theorems. II,
{\it Infinite and Finite Sets, Vol.~I,} 529--557.  Colloq.~Math.~Soc.~Janos
Bolyai, Vol.~10, North-Holland, Amsterdam, 1975.

\bibitem[4]{egmrss3}
Erd\H os, P., Graham, R.~L., Montgomery, P.,  Rothschild, B.~L., 
Spencer, J., and Straus, E.~G., Euclidean Ramsey Theorems. III,
{\it Infinite and Finite Sets, Vol.~I,} 559--583.  Colloq.~Math.~Soc.~Janos
Bolyai, Vol.~10, North-Holland, Amsterdam, 1975.

\bibitem[5]{triangles}
Frankl, P., and \rodl, V, All triangles are Ramsey, {\it 
Trans.~Amer.~Math.~Soc.} {\bf297} (1986), no. 2, 777--779.

\bibitem[6]{simplices}
Frankl, P., and \rodl, V., A partition property of simplices in
Euclidean space, {\it J.~Amer.~Math.~Soc.} {\bf3} (1990), no.~1, 1--7.

\bibitem[7]{conj}
Graham, R.~L., Recent trends in Euclidean Ramsey theory, {\it Discrete
Math.} {\bf136} (1994), no.~1--3, 119--127.

\bibitem[8]{book}
Graham, R.~L., Rothschild, B.~L., and Spencer, J.~H., Ramsey Theory,
John Wiley \& Sons, Inc., New York, 1980.

\bibitem[9]{hjt}
Hales, A.~W., and Jewett, R.~I., Regularity and positional games,
{\it Trans.~Amer.~Math.~Soc.} {\bf106} (1963), 222--229.

\bibitem[10]{johnson}
Johnson, F.~E.~A., Finite subtransitive sets, {\it Math.~Proc.~Cambridge
Philos.~Soc.} {\bf140} (2006), no.~1, 37--46.

\bibitem[11]{kriz}
\kriz, I., Permutation groups in Euclidean Ramsey Theory, 
{\it Proc.~Amer.~Math.~Soc.} {\bf112} (1991), no.~3, 899--907.

\bibitem[12]{trapezia}
\kriz, I., All trapezoids are Ramsey,  
{\it Discrete Math.} {\bf108} (1992), no.~1--3, 59--62.

\bibitem[13]{sard}
Milnor, J., Topology from the Differentiable Viewpoint,
The University Press of Virginia, Charlottesville, Va, 1965.

\bibitem[14]{vdw}
van der Waerden, B.~L., Beweis einer Baudet'schen Vermutung,
{\it Nieuw.~Arch.~Wisk.} {\bf15} (1927), 212--216.

\bibitem[15]{wolf}
Wolf, J.~A., Spaces of Constant Curvature, McGraw Hill, 1967.
\end{thebibliography}
\end{document}